\documentclass[a4paper,11pt]{amsart}

\usepackage{graphicx}%
\usepackage{multirow}%
\usepackage{amsmath,amssymb,amsfonts}%
\usepackage{amsthm}%
\usepackage{mathrsfs}%
\usepackage[title]{appendix}%
\usepackage{xcolor}%
\usepackage{textcomp}%
\usepackage{manyfoot}%
\usepackage{booktabs}%
\usepackage{algorithm}%
\usepackage{algorithmicx}%
\usepackage{algpseudocode}%
\usepackage{listings}%

\usepackage{amscd}
\usepackage{framed}
\usepackage[inline]{enumitem}
\usepackage{float}
\usepackage{multicol}
\usepackage{array}
\usepackage{makecell}
\usepackage{lscape}
\usepackage{times}
\usepackage{color}
\usepackage{cleveref}

\theoremstyle{plain}
\newtheorem{thm}{Theorem}[section]
\newtheorem{theorem}[thm]{Theorem}

\newtheorem{lemma}[thm]{Lemma}

\newtheorem{proposition}[thm]{Proposition}

\theoremstyle{definition}

\newtheorem{remark}[thm]{Remark}

\newtheorem*{claim*}{Claim}
\newtheorem*{claimproof*}{Proof of Claim}
\newtheorem*{notation*}{Notation}

\newtheorem{example}[thm]{Example}

\numberwithin{equation}{section}

\newcommand{\C}{{\mathbb C}}
\newcommand{\F}{{\mathbb F}}

\newcommand{\Q}{{\mathbb Q}}
\newcommand{\R}{{\mathbb R}}
\newcommand{\Z}{{\mathbb Z}}

\newcommand{\Amp}{\operatorname{Amp}}
\newcommand{\Aut}{\operatorname{Aut}}

\newcommand{\rank}{\operatorname{rank}}
\newcommand{\ord}{\operatorname{ord}}
\newcommand{\disc}{\operatorname{disc}}
\newcommand{\id}{\operatorname{id}}

\newcommand{\trace}{\operatorname{trace}}
\newcommand{\res}{\operatorname{res}}

\begin{document}

\title[Generalized Fibonacci numbers and automorphisms]{Generalized Fibonacci numbers and automorphisms of K3 surfaces with Picard number 2}

\author{Kwangwoo Lee}
\address{Department of Mathematics, Korea Institute for Advanced Study, 85 Hoegi-ro, Dongdaemun-gu, Seoul , Republic of Korea}
\email{klee0222@gmail.com}

\begin{abstract}
Using the properties of generalized Fibonacci numbers, we determine the automorphism groups of some K3 surfaces with Picard number $2$. Conversely, using the automorphisms of K3 surfaces with Picard number $2$,   we prove the criterion for a given integer $n$ is to be a generalized Fibonacci number. Moreover, we show that the generalized $k$-th Fibonacci number divides the generalized $q$-th Fibonacci number if and only if $k$ divides $q$.
\end{abstract}

\thanks{The author would like to thank Professor Yuta Takada for letting me to read his preprint. The author was supported by National Research Foundation of Korea(NRF-2022R1A2B5B03002625) grant funded by the Korea government(MSIT).}


\maketitle

\section{Introduction}
A \textit{K3 surface} is a simply connected compact manifold $X$ such that $H^0(X,\Omega_X^2)=\mathbb{C}\omega_X$, where $\omega_X$ is an everywhere non-degenerate holomorphic $2$-form on $X$. In \cite{O1}, K. Oguiso considered a K3 surface $S$ with Picard lattice $NS(S)=\Z h_1\oplus \Z h_2$ where the intersection matrix
\begin{equation}\label{Picard lattice of Oguiso}
[(h_i,h_j)]=\begin{bmatrix} 4& 2\\ 2&-4\end{bmatrix}.
\end{equation}
He constructed an isometry of $NS(S)$ given by the multiplication of $\eta^6$ on $NS(S)$, where $h_2=\eta:=\frac{1+\sqrt{5}}{2}$ is a basis of the lattice $NS(S)$ together with $h_1:=1$. Then by the global Torelli theorem for K3 surfaces, it induces an automorphism $g$ of $S$ with positive entropy without fixed points. It is also known as the \textit{Cayley-Oguiso} automorphism in \cite{FGvGvL}. In \cite{L1}, the author generalized Oguiso's construction by considering Picard lattice $NS(X)$ whose intersection matrix
\begin{equation}\label{Picard lattice}
L_m(a):=m\begin{bmatrix} 2& a\\ a&-2\end{bmatrix}.
\end{equation} Moreover, in \cite[Theorem 1]{L1}, the author showed that (1) the isometry group of $L_m(a)$ is $O(L_m(a))\cong \Z_2\ast \Z_2$, a free product with generators 
\begin{equation}\label{A and B}
A=\begin{bmatrix} 1& 0\\ a&-1\end{bmatrix} \hspace{2.6mm} \text{ and } \hspace{2.6mm} B=\begin{bmatrix} 1& a\\ 0&-1\end{bmatrix},
\end{equation}
 and (2) $\Aut(X)\cong \Z$ for $m\geq 2$. Note that the generator $g$ of $\Aut(X)$ acts on $NS(X)$ as $(AB)^n$ for some $n$. In this paper we determine $n$ for a given $m$. 
\begin{theorem}\label{main theorem1} Let $X$ be a K3 surface with Picard lattice $NS(X)$ whose intersection matrix is $L_m(a)$ ($m\geq 2$) in \eqref{Picard lattice}. 
\begin{enumerate}
\item If $n$ is an even number such that $m\mid a_n$, the generalized $n$-th Fibonacci number (cf. Section \ref{generalized Fibonacci numbers}), then $X$ has a symplectic automorphism $g$ such that $g^*\vert_{NS(X)}=(AB)^n$. Moreover, if $5\nmid n$ and $n$ is minimal then $\Aut(X)=\langle g\rangle$, that is, $g$ is a generator and its characteristic polynomial $g^*\vert_{H^2(X,\Z)}=S(x)(x-1)^{20}$, where $S(x)$ is the characteristic polynomial of $(AB)^n$.
\item If $n$ is an odd number such that $m\mid a_n$, then $X$ has an anti-symplectic automorphism $g$ such that $g^*\vert_{NS(X)}=(AB)^n$. Moreover, if $5\nmid n$ and $n$ is minimal then $\Aut(X)=\langle g\rangle$, that is, $g$ is a generator and its characteristic polynomial $g^*\vert_{H^2(X,\Z)}=S(x)(x+1)^{20}$, where $S(x)$ is the characteristic polynomial of $(AB)^n$. 
\end{enumerate}
\end{theorem}

\begin{remark} The generators in (1) and (2) are symplectic and anti-symplectic, respectively, i.e., $g^*\omega_X=\omega_X$ and $g^*\omega_X=-\omega_X$. $S(x)$ is the Salem polynomial (cf. Section \ref{Aut of K3}) of $g$ with the Salem  trace $(a^2+4)a_n^2+(-1)^n2$ (cf. Lemma \ref{trace of nth power}), i.e., $S(x)=x^2-((a^2+4)a_n^2+(-1)^n2)x+1$.
\end{remark}

Let $\{f_n\}_{n\geq 0}$ be the Fibonacci sequence with $f_{n+2}=f_{n+1}+f_n$ and $f_0=0, f_1=1$. In \cite{W}, R.E. Whitney showed that a given integer $n$ is a Fibonacci number if and only if either $5n^2+4$ or $5n^2-4$ is a square number. For a \textit{generalized Fibonacci sequence} $\{a_n\}_{n\geq 0}$ with $a_{n+2}=a\cdot a_{n+1}+a_n$ with $a_0=0, a_1=1$, and a fixed $a\in \Z_{\geq 1}$, we also show:

\begin{theorem}\label{main theorem2}
$n$ is the $k$-th generalized Fibonacci number $a_k$ and $k$ is even (resp. odd) if and only if $(a^2+4)n^2+4$   (resp. $(a^2+4)n^2-4$) is a perfect square.
\end{theorem}
\begin{remark} For $a=1$, we get Whitney's result. In particular, we can determine which of $5n^2+4$ or $5n^2-4$ is a square number for a given $n$.
\end{remark}


We also show:
\begin{theorem} \label{main theorem3}
 For positive integers $k<q$, 
\begin{enumerate}
\item $a_k$ and $a_{k+1}$ are relatively primes, i.e., $\gcd(a_k,a_{k+1})=1$.
\item If $a_k\mid a_q$, then $a_k\mid a_{q-k}$.
\item $k\mid q$  if and only if $a_k\mid a_q$.
\end{enumerate}
\end{theorem}
\hfill \break
\begin{notation*} 
\hspace{2mm}
\begin{itemize}
\item $f_n$ is denoted for the original Fibonacci numbers in the paper. 
\item $a_n$ is denoted for the generalized Fibonacci numbers in the paper. 
\item $X_{L_m(a)}$ is denoted a K3 surface $X$ whose intersection matrix of Picard lattice is given by \eqref{Picard lattice}.
\end{itemize}
\end{notation*}

\section{Preliminaries}
\subsection{Generalized Fibonacci numbers}\label{generalized Fibonacci numbers}
For $a\in \Z_{\geq 1}$, let $\{a_n\}_{n\geq 0}$ be the generalized Fibonacci numbers satisfying 
\begin{equation}\label{gen Fibonacci seq}
a_{n+2}=a\cdot a_{n+1}+a_n, \text{ with }a_0=0, a_1=1.
\end{equation}
If $a=1$, we get the original Fibonacci sequence $\{f_n\}$.

For $A$ and $B$ in \eqref{A and B}, by \cite[Section 2.1]{L1},
\begin{equation}
(AB)^n=\begin{bmatrix} a_{2n-1}& a_{2n}\\ a_{2n}&a_{2n+1}\end{bmatrix}.
\end{equation}

\begin{lemma}\cite[Lemma 9]{L1}\label{properties of Fibonacci seq}
For the sequence in \eqref{gen Fibonacci seq}, 
\begin{enumerate}
\item $a_{n+k}=a_ka_{n+1}+a_{k-1}a_n=a_{k+1}a_n+a_ka_{n-1}$.
\item $a_{n+1}a_{n-1}-a_n^2=\left\{ \begin{array}{ll}
      1 & \mbox{if $n$ is even},\\
      -1 & \mbox{if $n$ is odd}.\end{array} \right.$ 
\end{enumerate}
\end{lemma}

\begin{lemma}\label{trace of nth power}
$\trace((AB)^n)=a_{2n-1}+a_{2n+1}=(a^2+4)a_n^2+(-1)^n2$.
\end{lemma}
\begin{proof}
By Lemma \ref{properties of Fibonacci seq},
\begin{align}
\begin{split}
a_{2n+1}+a_{2n-1}&=(a_{n+1}^2+a_n^2)+(a_n^2+a_{n-1}^2)=2a_n^2+a_{n+1}^2+a_{n-1}^2\\
&=2a_n^2+(a\cdot a_n+a_{n-1})^2+a_{n-1}^2\\
&=(a^2+2)a_n^2+2a_{n-1}(a\cdot a_n+a_{n-1})=(a^2+2)a_n^2+2a_{n-1}a_{n+1}\\
&=(a^2+2)a_n^2+2(a_n^2+(-1)^n)=(a^2+4)a_n^2+(-1)^n2.\qedhere
\end{split}
\end{align}
\end{proof}

\begin{lemma}\label{shifted trace of nth power}
$a_{2n-2}+a_{2n}=((a^2+4)(a_{n}^2-a_{n-1}^2)+(-1)^n4)/a.$
\end{lemma}
\begin{proof} 
By Lemma \ref{trace of nth power},
\begin{align}
\begin{split}
a_{2n}+a_{2n-2}&=(a_{2n+1}-a_{2n-1})/a+(a_{2n-1}-a_{2n-3})/a\\
&=(a_{2n+1}+a_{2n-1})/a-(a_{2n-1}+a_{2n-3})/a\\
&=((a^2+4)a_n^2+(-1)^n2)/a-((a^2+4)a_{n-1}^2+(-1)^{n-1}2)/a\\
&=((a^2+4)(a_n^2-a_{n-1}^2)+(-1)^n4)/a.\qedhere
\end{split}
\end{align}
\end{proof}

Note that for the original Fibonacci numbers $f_n$ with $a=1$, we have that $f_{2n-1}+f_{2n+1}=5f_n^2+(-1)^n2$.

\subsection{Resultants}\label{Resultants}
For two polynomials $P, Q\in \Z[x]$ without common zeros in $\C$, the \textit{resultant} of $P$ and $Q$ is the integer given by
\begin{equation}\label{resultant}
\res\{P, Q\}=\prod_{P(u)=0,Q(v)=0}(u-v).
\end{equation}

We use the fact that $\bar{P}(x),\bar{Q}(x)$ have common factors in $\F_p[x]$ if and only if $p\mid \res\{P, Q\}$, where $P\mapsto \bar{P}$ under $\Z[x]\rightarrow \F_p[x]$.

\subsection{Lattices}\label{Lattices}

A \textit{lattice} is a free $\Z$-module $L$
 of finite rank equipped with a symmetric bilinear form
 $\langle ~,~ \rangle \colon L\times L\rightarrow \Z$.
If $x^2:=\langle x,x\rangle \in 2\Z$ for any $x\in L$,
 a lattice $L$ is said to be even. 
We fix a $\Z$-basis of $L$ and identify the lattice $L$
 with its intersection matrix $Q_L$ under this basis.
The \textit{discriminant} $\disc(L)$ of $L$ is defined as $\det(Q_L)$,
 which is independent of the choice of basis.
A lattice $L$ is called \textit{non-degenerate} if
 $\disc(L)\neq 0$ and \textit{unimodular} if $\disc(L)=\pm1$.
For a non-degenerate lattice $L$,
 the signature of $L$ is defined as $(s_+,s_-)$,
 where $s_{+}$ (resp.\ $s_-$)
 denotes the number of the positive (resp.\ negative) eigenvalues of
 $Q_L$.
An \textit{isometry} of $L$ is an automorphism of the $\Z$-module $L$
 preserving the bilinear form.
The \textit{orthogonal group} $O(L)$ of $L$ consists of the isometries of $L$ and we have
 the following identification:
\begin{equation} \label{EQ_ortho}
 O(L)= \{ g\in GL_n(\Z) \bigm| g^T \cdot Q_L \cdot g=Q_L \}, \quad
 n=\rank L.
\end{equation}
For a non-degenerate lattice $L$,
 the \textit{discriminant group} $A(L)$ of $L$ is defined by
\begin{equation} 
 A(L):=L^*/L, \quad
 L^*:=
 \{ x\in L\otimes \Q \bigm|
 \langle x,y \rangle \in \Z ~ \forall y\in L \}.
\end{equation}

Let $K$ be a sublattice of a lattice $L$,
 that is, $K$ is a $\Z$-submodule of $L$
 equipped with the restriction of the bilinear form of $L$ to $K$.
If $L/K$ is torsion-free as a $\Z$-module,
 $K$ is said to be \textit{primitive}.

For a non-degenerate lattice $L$ of signature $(1,k)$ with $k\geq 1$,
 we have the decomposition
\begin{equation} \label{EQ_positive_cone}
 \{x\in L\otimes \R \bigm| x^2 >0 \}=C_L\sqcup (-C_L)
\end{equation}
 into two disjoint cones.
Here $C_L$ and $-C_L$ are connected components and $C_L$ is called the \textit{positive cone}.
We define
\begin{equation} \label{O^+}
 O^+(L):=\{g\in O(L) \bigm| g(C_L)=C_L \}, \quad
 SO^+(L):=O^+(L)\cap SO(L),
\end{equation}
 where $SO(L)$ is the subgroup of $O(L)$
 consisting of isometries of determinant $1$.
The group $O^+(L)$ is a subgroup of $O(L)$ of index $2$.

\begin{lemma}\cite[Lemma 1]{HKL}\label{criterion}
Let $L$ be a non-degenerate even lattice of rank $n$. For $g\in O(L)$ and $\epsilon\in\{\pm 1\}$, $g$ acts on $A(L)$ as $\epsilon\cdot \id$ if and only if $(g-\epsilon\cdot I_n)\cdot Q_L^{-1}$ is an integer matrix. 
\end{lemma}

\subsection{Automorphisms of K3 surfaces}\label{Aut of K3}
By the following proposition, we have a K3 surface with Picard lattice whose intersection matrix is given by \eqref{Picard lattice}.
\begin{proposition}\cite[Proposition 2.]{HKL} \label{existence of K3}
For any even lattice $L$ of signature $(1,1)$, there exists a projective K3 surface $X$ such that $NS(X)\cong L$.
\end{proposition}

Let $X$ be a K3 surface with Picard lattice $NS(X)$ whose intersection matrix is given in \eqref{Picard lattice}.
Since $X$ has no $(-2)$-curve, the ample cone of $X$ is just the positive cone, i.e.,
\begin{equation}\label{ample cone of X}
\begin{array}{ll}
\Amp(X)&=\{(x,y)\in NS(X)\otimes \R\mid (x,y)^2>0 \text{ and } x>0\}\\
&=\{(x,y)\in NS(X)\otimes \R\mid x>0 \text{ and }\frac{a-\sqrt{a^2+4}}{2}<\frac{y}{x}<\frac{a+\sqrt{a^2+4}}{2}\}.\\
\end{array}
\end{equation}

Let $g$ be an automorphism of a K3 surface $X$. For the induced isometry $g^*:H^2(X,\C)\rightarrow H^2(X,\C)$, let $\lambda(g)$ be the spectral radius:
\begin{equation}\label{spectral radius}
\lambda(g):=\max\{\vert \mu\vert \mid \mu\in \C \text{ is an eigenvalue of } g^*\}.
\end{equation}
It is known that $\lambda(g)$ is equal to $1$ or a \textit{Salem number}, i.e., a real algebraic number $\lambda>1$ whose conjugates other than $\lambda^{\pm 1}$ lie on the unit circle. By the fundamental result of Gromov and Yomdin (\cite{Gr, Yo}), the topological entropy of an automorphism of a compact K\"ahler surface is the logarithm of a Salem number. 
 The irreducible monic integer polynomial $S(x)$ of $\lambda$ is a \emph{Salem polynomial} and the degree of $S(x)$ is the degree of Salem number $\lambda$. A \emph{Salem trace} is an algebraic integer $\tau>2$ whose other conjugates lie in $(-2,2)$; the minimal polynomial of such a number is called a \emph{Salem trace polynomial}. Salem traces and Salem numbers correspond bijectively, via the relation $\tau=\lambda+\lambda^{-1}$.

All Salem polynomials satisfy the equation $x^nS(\frac{1}{x})=S(x)$, where $n=\deg S(x)$. This simply means that its coefficients form a palindromic sequence, that is, they read the same backwards as forwards.

For the characteristic polynomials $F(x)$ and $S(x)$ of $g^*:H^2(X)\rightarrow H^2(X)$ and $g^*\vert_{NS(X)} :NS(X)\rightarrow NS(X)$, $F(x)=S(x)\Phi_l(x)^m$
for some $l$ and $m\in \Z_{\geq 1}$ by \cite[Corollary 8.13]{K}, where $\Phi_l(x)$ is the $l$-th cyclotomic polynomial. 
\begin{proposition}\cite[Proposition 3.6]{T}\label{cyclotomic polynomial}
For Picard number $2$ with Salem polynomial $S(x)$ of degree $2$, $l=1,2,5,10,25$, or $50$. Moreover, putting
\begin{align}
\begin{split}
\epsilon=\left\{ \begin{array}{ll}
       1 & \mbox{if $l=1, 5, 25$},\\
       -1 & \mbox{if $l=2, 10, 50$},\end{array} \right. 
\end{split}
\end{align}
the integer $\tau+2\epsilon$ is a square number, and if $l\neq 1, 2$, then $5(\tau-2\epsilon)$ is a square number.
\end{proposition}

\section{Proofs}
We prove Theorems \ref{main theorem1}, \ref{main theorem2}, and \ref{main theorem3} in this section.

\subsection{Proof of Theorem \ref{main theorem1}}
Let $X$ be a K3 surface whose intersection matrix of Picard lattice $NS(X)$ is given by $L_m(a)$ ($m\geq 2$) in \eqref{Picard lattice}. Proposition \ref{existence of K3} guaranties the existence of such $X$. If $((AB)^n-I_2)Q_{L_m(a)}^{-1}$ is an integer matrix, by Lemma \ref{criterion}, $(AB)^n$ is an isometry of $NS(X)$ acting on $A(NS(X))$ as $\id$. Now by gluing the isometries $(AB)^n$ on $NS(X)$ and $\id$ on $T_X:=NS(X)^{\perp}_{\Lambda}$, we have an isometry $\phi$ of $\Lambda\cong H^2(X,\Z)$ such that $\phi_{NS(X)}=(AB)^n$ and $\phi_{T_X}=\id$, where $\Lambda$ is the K3 lattice. Furthermore by Torelli theorem for K3 surfaces, $\phi$ extends to an automorphism $g$ of $X$ since it preserves the ample cone of $X$. Note that the ample cone of $X$ is just the positive cone since $NS(X)$ has no $(-2)$ vectors.

Now suppose that the following is an integer matrix:
\begin{align}
\begin{split}\label{eq:1}
       ((AB)^n&-I_2)L_m^{-1}=\begin{bmatrix} a_{2n-1}-1& a_{2n}\\ a_{2n}&a_{2n+1}-1\end{bmatrix}\frac{1}{m(a^2+4)}\begin{bmatrix} 2& a\\ a&-2\end{bmatrix}\\
       &=\frac{1}{m(a^2+4)}\begin{bmatrix} a_{2n-1}+a_{2n+1}-2& -(a_{2n-2}+a_{2n})-a\\ a_{2n}+a_{2n+2}-a&-(a_{2n-1}+a_{2n+1})+2\end{bmatrix}\\
       &=\frac{1}{m(a^2+4)}\begin{bmatrix} (a^2+4)(a_{n})^2+(-1)^n2-2& \frac{-(a^2+4)(a_n^2-a_{n-1}^2)-(-1)^n4-a^2}{a}\\ \frac{(a^2+4)(a_{n+1}^2-a_{n}^2)+(-1)^{n+1}4-a^2}{a}& -(a^2+4)(a_{n})^2-(-1)^n2+2\end{bmatrix}.
\end{split}
\end{align}
If  $n$ is odd, then it cannot be an integer matrix, hence $n$ should be even. Moreover, $m$ should divide $a_n^2, (a_n^2-a_{n-1}^2+1)/a$, and $(a_{n+1}^2-a_n^2-1)/a$. By Lemma \ref{properties of Fibonacci seq} (2), $a_{n-1}^2-1=a_na_{n-2}$ and $a_{n+1}^2-1=a_{n+2}a_n$. Hence $(a_n^2-a_{n-1}^2+1)/a=a_na_{n-1}$ and $(a_{n+1}^2-a_n^2-1)/a=a_na_{n+1}$. So $m\mid a_n$ implies that the matrix is an integer matrix.

Now if $5\nmid n$, then $g^*\vert_{NS(X)}=(AB)^n$ is a symplectic automorphism by Proposition \ref{cyclotomic polynomial}. Moreover, if $m\mid a_{n'}$, the generalized $n'$-th Fibonacci number, then $h^*\vert_{NS(X)}=(AB)^{n'}$ for some $h\in \Aut(X)$. Since $\Aut(X)\cong \Z$, if $n\leq n'$, then $h=g^k$ for some $k$, hence $n\mid n'$. This implies that if $5\nmid n$ and $n$ is minimal such that $m\mid a_n$, then $g$ with $g^*\vert_{NS(X)}=(AB)^n$ is the symplectic generator of $\Aut(X)\cong \Z$. This proves (1) of Theorem \ref{main theorem1}.

By the similar arguments, we can show that $((AB)^n+I_2)Q_{L_m}^{-1}$ is an integer matrix if $n$ is odd and $m\mid a_n$. Moreover, if $5\nmid n$ and $n$ is minimal such that $n\mid a_n$, then $g$ with $g^*\vert_{NS(X)}=(AB)^n$ is the anti-symplectic generator of $\Aut(X)\cong \Z$.  This proves (2) of Theorem \ref{main theorem1}.
\qed


\subsection{Proof of Theorem \ref{main theorem2}}

Suppose that $n$ is the $k$-th generalized Fibonacci number $a_k$. If $k$ is even, by Lemma \ref{trace of nth power}, $(a^2+4)a_k^2+2=a_{2k-1}+a_{2k+1}$ and which is $\trace (AB)^k$.  Now by Theorem \ref{main theorem1} (1), $(AB)^k$ is an isometry of $NS(X)$ for a symplectic automorphism $g$ of a K3 surface $X_{L_m(a)}$, where $m(>1)$ is any divisor of $a_k$. Then by \cite[Main Theorem]{HKL}, $(a^2+4)a_k^2+2=\alpha^2-2\epsilon$ for some $\alpha\in A_{\epsilon}$ and $\epsilon=\pm 1$, where 
\begin{equation}
A_{\epsilon}=\left\{ \begin{array}{ll}
       \Z_{\geq 4} & \mbox{if $\epsilon=1$},\\
       \Z_{\geq 4}\setminus \{5,7,13,17\} & \mbox{if $\epsilon=-1$}.\end{array} \right. 
\end{equation}
Since  $g$ is symplectic, $\epsilon=1$ and $(a^2+4)a_k^2+4=\alpha^2$ for some $\alpha>3$. 

Similarly if $k$ is odd, then $(a^2+4)a_k^2-2=a_{2k-1}+a_{2k+1}$ and which is $\trace (AB)^k$. By Theorem \ref{main theorem1} (2), $(AB)^k$ is an isometry of $NS(X)$ for an anti-symplectic automorphism $g$ of a K3 surface $X_{L_m(a)}$, where $m(>1)$ is any divisor of $a_k$. Then by \cite[Main Theorem]{HKL}, $(a^2+4)a_k^2-2=\alpha^2-2\epsilon$ for some $\alpha\in A_{\epsilon}$. In this case, $g$ is anti-symplectic, hence $\epsilon=-1$ and $(a^2+4)a_k^2-4=\alpha^2$ for some $\alpha>3$. 


Conversely, if $(a^2+4)n^2\pm 4=\alpha^2$ with $\alpha>3$, then by \cite[Main Theorem]{HKL}, $\alpha^2-2\epsilon=(a^2+4)n^2\pm 4-2\epsilon$ is a trace of some automorphism $g$ of a projective K3 surface $X$ with Picard number $2$ such that $\ord(g)=\infty$ and $g^*\omega_X=\epsilon \omega_X$, where $\epsilon=\pm 1$. Now by \cite[Lemma 3]{HKL}, $\alpha^2-2\epsilon$ is a trace of an isometry acting on $A(L_m(a))$ as $\epsilon\cdot \id$ if and only if we have $(u,v)=(\alpha^2-2\epsilon, \alpha\beta)$ with $\alpha^2-(a^2+4)m^2\beta^2=4\epsilon$. Indeed if we let $m=n$ and $\beta=1$, then $\alpha^2-2\epsilon=(a^2+4)n^2\pm 4-2\epsilon$ is a trace of an isometry $\phi$ of a lattice whose intersection matrix is given by
\begin{equation}
L_n(a)=n\begin{bmatrix} 2& a\\ a&-2\end{bmatrix}.
\end{equation} 
Since $\alpha>3$, we know that $n>1$, hence the K3 surface $X$ with $L_n(a)$ as the intersection matrix of Picard lattice $NS(X)$ has no $(-2)$ vectors. Now by gluing the isometries $\phi$ on $NS(X)$ and $\epsilon\cdot \id$ on $T(X):=NS(X)^{\perp}$, they extend to an isometry of $H^2(X,\Z)$. By Torelli theorem, it extends to an automorphism of the K3 surface $X_{L_n(a)}$. 
Moreover, any automorphism of $X_{L_n(a)}$ is acting as $(AB)^k$ on $NS(X)$ for some $k$ (cf. Introduction). Hence the trace $(a^2+4)n^2\pm4-2\epsilon=a_{2k-1}+a_{2k+1}$. Now by Lemma \ref{trace of nth power}, $n=a_k$ a generalized Fibonacci number. \qed

\begin{remark}\label{remark1}
Let $a=1$ and $m\geq 2$ in \eqref{Picard lattice}. For the generator $g\in \Aut(X_{L_m(1)})$ with $g^*\vert_{H^2(X_{L_m(1)},\Z)}=(AB)^n$, if $l=1, 5,$ or $25$, then $n$ should be even and if $l=2, 10,$ or $50$, $n$ should be odd. Because otherwise $5(5f_n^2-2-2)$ or $5(5f_n^2+2+2)$ cannot be a square number by Theorem \ref{main theorem2}.
\end{remark}

\subsection{Proof of Theorem \ref{main theorem3}}

(1) For $k=1$, it is obvious. Let $k\in \Z_{>1}$. Suppose that a prime number $p$ divides both $a_k$ and $a_{k+1}$. By Theorem \ref{main theorem1}, if $m=p$, then the K3 surface $X_{L_{p}(a)}$ has a generator $g\in \Aut(X_{L_{p}(a)})$ such that $g^{\alpha *}\vert_{NS(X_{L_{p}(a)})}=(AB)^k$ and $g^{\beta *}\vert_{NS(X_{L_{p}(a)})}=(AB)^{k+1}$ for some $\alpha, \beta\in \Z_{>0}$ and one is symplectic and the other one is anti-symplectic. Hence the generator $g$ acts on $NS(X_{L_{p}(a)})$ as $g^*\vert_{NS(X_{L_{p}(a)})}=(AB)$. However, since one of $g^k$ and $g^{k+1}$ is symplectic and the other one is anti-symplectic, $k=1$ and $g$ is anti-symplectic, which is a contraction to $k> 1$. 
\\

\noindent (2) Suppose that $a_k\mid a_q$. By Lemma \ref{properties of Fibonacci seq} (1), $a_q=a_{k+q-k}=a_{q-k}a_{k+1}+a_{q-k-1}a_k$. Hence $a_k\mid a_{q-k}a_{k+1}$, but by (1), $a_k\mid a_{q-k}$. 
\\

\noindent (3) Suppose that $k\mid q$, i.e., $q=kr$ for some $r$. Then by Theorem \ref{main theorem1}, for $m=a_k$, there is a K3 surface $X_{L_{a_k}(a)}$ with $g\in \Aut(X_{L_{a_k}(a)})$ such that $g^*\vert_{L_{a_k}(a)}=(AB)^k$. Then $g^{r*}\vert_{L_{a_k}(a)}=(AB)^q$. Since either $g^r$ is symplectic or anti-symplectic, by Lemma \ref{criterion} and \eqref{eq:1}, $a_k(=m)$ divides $a_q$.

For the other direction, suppose that $a_k\mid a_q$. Then by (2), $a_k\mid a_{q-k}$. Repeating this argument, we see that $k\mid q$.
\qed

\section{Examples}\label{Examples}

\begin{example}\label{example1}
If $m=3$ and $a=1$, then $3\mid f_4=3$ and hence $g^*\vert_{NS(X)}=(AB)^4$ is the isometry of the symplectic generator $g$ of $\Aut(X)$, where the intersection matrix of $NS(X)=3\begin{bmatrix} 2& 1\\ 1&-2\end{bmatrix}.$ Similarly, $13\mid f_7=13$, hence for a K3 surface $X$ with $NS(X)$ whose intersection matrix $=13\begin{bmatrix} 2& 1\\ 1&-2\end{bmatrix}$, $\Aut(X)=\langle g\rangle\cong \Z$ where $g^*\vert_{NS(X)}=(AB)^7$  and $g$ is anti-symplectic. 
\end{example}

\begin{example}
Suppose that an integer $m\neq 5 (>1)$ divides $f_{20}=3\cdot 5\cdot 11\cdot 41$, the $20$-th Fibonacci number. Then by Theorem \ref{main theorem1} (1), there is a K3 surface $X$ with
\begin{equation}
m\begin{bmatrix} 2& 1\\ 1&-2\end{bmatrix}
\end{equation}
as the intersection matrix of Picard lattice $NS(X)$ such that $g^*\vert_{NS(X)}=(AB)^{20}$, where $g$ is a symplectic automorphism. If $h$ is a generator of $\Aut(X)$ with $h^*_{NS(X)}=(AB)^k$ and $h^*\omega_X=\zeta_l\omega_X$, then $k\mid 20$, i.e., $k=1, 2, 4, 5, 10,$ or $20$ and $l=1, 2,$ or $10$. Note that since $g$ is symplectic, $l=25, 50$ cannot occur. If $l=10$, then $k=2$, but which is not possible by Remark \ref{remark1}. If $l=2$, then $h$ is anti-symplectic, hence $k=1$ or $5$ by again Remark \ref{remark1}. Moreover, by Theorem \ref{main theorem1} (2), $m\mid f_1=1$ or $f_{5}=5$. Since $m\neq 5 (>1)$, $l=2$ cannot occur, hence $g$ is the generator of $\Aut(X)$. 
\end{example}

\begin{example}\label{example 100}
Suppose that an integer $m(>1)$ divides $f_{100}=3\cdot 5^2\cdot 11\cdot 41\cdot 101\cdot 151\cdot 401\cdot 3001\cdot 570601$, the $100$-th Fibonacci number, but $m\nmid f_{50}=5^2\cdot 11\cdot 101\cdot 151\cdot 3001$, the $50$-th Fibonacci number. Then by Theorem \ref{main theorem1} (1), there is a K3 surface $X$ with
\begin{equation}
m\begin{bmatrix} 2& 1\\ 1&-2\end{bmatrix}
\end{equation}
as the intersection matrix of Picard lattice $NS(X)$ such that $g^*\vert_{NS(X)}=(AB)^{100}$, where $g$ is a symplectic automorphism. If $h\in \Aut(X)$ is a generator of $X$ with $h^*_{NS(X)}=(AB)^k$ with $h^*\omega_X=\zeta_l\omega_X$, then $k\mid 100$, i.e., $k=1, 2, 4, 5, 10, 20, 25, 50,$ or $100$ and $l=1, 2, 10, 25,$ or $50$. $l=50$ cannot occur because $k=2$ is not possible by Remark \ref{remark1} and $k=1$ cannot occur because $m\nmid f_{50}$. If $l=25$, then $k=2$ or $4$, but $k=2$ cannot occur because $m\nmid f_{50}$. $l=10$ cannot occur because $m\nmid f_{50}$. If $l=5$, then $k=2, 4, 10,$ or $20$ by Remark \ref{remark1}, but $k=2,10$ cannot occur because $m\nmid f_{50}$. $l=2$ cannot occur because $m\nmid f_{50}$. If $l=1$, then $k=2, 4, 10, 20, 50,$ or $100$, but $k=2, 10, 20, 50$ cannot occur because $m\nmid f_{50}$. Hence, for the generator $h$ of $\Aut(X)$, $h^*_{NS(X)}=(AB)^k$ and $h^*\omega_X=\zeta_l\omega_X$, $(l,k)=(25,4), (5,4),  (5,20), (1,4)$ or $(1,100)$. Note that indeed $(l,k)=(1,4)$ occurs in Example \ref{example1}.
\end{example}

\section{Case $5\mid m$}
In \cite{T}, Y. Takada computed all possible Salem traces of projective K3 surfaces with Picard number $2$. In particular, he constructed an automorphism $g$ with $g^*\vert_{NS(X)}=AB$ of some K3 surface $X$, where $a=1$ in \eqref{Picard lattice}. In order the trace $3$ of $(AB)$ to be a Salem trace of some automorphism $g$ of a K3 surface with Picard number $2$, by Proposition \ref{cyclotomic polynomial} and \cite[Main Theorem]{HKL}, $g^*\omega_X=\zeta_l\omega_X$ where  $\zeta_l$ is the primitive $l$-th root of unity with $l=5,10,25$, or $50$. That is, $g$ cannot be symplectic or anti-symplectic. 
Moreover, by \cite[Corollary 8.13]{K}, the characteristic polynomial of $g^*_{H^2(X,\Z)}=S(x)(\phi_l(x))^m$ for some $l$ and $m$.  By \cite[Ch.14. Lemma 2.5]{H}, $\bar{S}(x)=\bar{\phi_l}^m(x)$ for any prime divisor $p$ of $\vert \disc(NS(X))\vert$, where $P\mapsto \bar{P}$ under $\Z[x]\rightarrow \F_p[x]$.
For a characteristic polynomial $F(x)=S(x)(\phi_l(x))^m$ of an automorphism $g$ on $H^2(X,\Z)$, $S(x)$ and $\phi_l(x)$ have a common factor modulo any prime divisor $p_i$ of $\vert \disc(NS(X))\vert$, where $A(NS(X))\cong A(T_X)\cong \otimes \Z_{p_i}^{m_i}$. Hence any prime divisor $p_i$ of $\disc(NS(X))$ divides the resultant of $S(x)$ and $\phi_l(x)$.


\begin{lemma}\label{resultant with cyclotomic polynomials}
For $\lambda=\frac{1+\sqrt{5}}{2}$, 
\begin{enumerate}
\item 
\begin{equation*}
\res\{(x-\lambda^n)(x-\lambda^{-n}),\phi_{5}\}
=\left\{ \begin{array}{ll}
       5^2(5f_{n}^4+5f_{n}^2+1)^2 & \mbox{if $n$ is even},\\
       (5^2f_{n}^4-15f_{n}^2+1)^2 & \mbox{if $n$ is odd}.\end{array} \right. 
\end{equation*}
\item
\begin{equation*}
\res\{(x-\lambda^n)(x-\lambda^{-n}),\phi_{10}\}=\left\{ \begin{array}{ll}
       (5^2f_{n}^4+15f_{n}^2+1)^2 & \mbox{if $n$ is even},\\
        5^2(5f_{n}^4-5f_{n}^2+1)^2 & \mbox{if $n$ is odd}.\end{array} \right. 
\end{equation*}
\item
\begin{equation*}
\res\{(x-\lambda^n)(x-\lambda^{-n}),\phi_{25}\}=\left\{ \begin{array}{ll}
        5^2(5f_{5n}^4+5f_{5n}^2+1)^2 & \mbox{if $n$ is even},\\
       (5^2f_{5n}^4-15f_{5n}^2+1)^2 & \mbox{if $n$ is odd}.\end{array} \right. 
\end{equation*}
\item
\begin{equation*}
\res\{(x-\lambda^n)(x-\lambda^{-n}),\phi_{50}\}=\left\{ \begin{array}{ll}
      (5^2f_{5n}^4+15f_{5n}^2+1)^2 & \mbox{if $n$ is even},\\
        5^2(5f_{5n}^4-5f_{5n}^2+1)^2 & \mbox{if $n$ is odd}.\end{array} \right. 
\end{equation*}
\end{enumerate}
\end{lemma}
\begin{proof} Let $\mu=\lambda^n$. For $\phi_5(x)=x^4+x^3+x^2+x+1$ and $(x-\mu)(x-\mu^{-1})$, 
\begin{align}
\begin{split}
\res\{(x-\lambda^n)(x-\lambda^{-n}),\phi_{5}\}&=\prod_{1\leq i\leq 4}(\mu-\zeta_i)\prod_{1\leq i\leq 4}(\mu^{-1}-\zeta_i)=\phi_5(\mu)\phi_5(\mu^{-1})\\
&=(\mu^4+\mu^3+\mu^2+\mu+1)(\mu^{-4}+\mu^{-3}+\mu^{-2}+\mu^{-1}+1)\\
&=(\frac{\mu^4+\mu^3+\mu^2+\mu+1}{\mu^2})^2=(\mu^2+\mu^{-2}+\mu+\mu^{-1}+1)^2,
\end{split}
\end{align}
where $\zeta_i$ is the primitive $5$-th roots of unity, i.e., the roots of $\phi_5$.
Since $\mu=\lambda^n$, $\mu+\mu^{-1}=\lambda^n+\lambda^{-n}=\trace(AB)^n$ for $A$ and $B$ in \eqref{A and B} with $a=1$. Hence by Lemma \ref{trace of nth power}, we have (1).

By the similar arguments, we have (2), (3), and (4).
\end{proof}

By the argument above Lemma \ref{resultant with cyclotomic polynomials}, we have the following:
\begin{proposition}
Let $X$ be a K3 surface with intersection matrix \eqref{Picard lattice} of Picard lattice. For an isometry $g$ of  $H^2(X,\Z)$  with characteristic polynomial $g^*_{H^2(X,\Z)}=S(x)(\phi_l(x))^m$, if some divisor of $m$ does not divide $\res\{S(x), \phi_l(x)\}$, then $g$ cannot be extended to an automorphism of $X$.
\end{proposition}

\begin{example} If $n=1$ in Lemma \ref{resultant with cyclotomic polynomials}, then $\res\{x^2-3x+1,\phi_{5}\}=11^2$, $\res\{x^2-3x+1,\phi_{10}\}=5^2,$ $\res\{x^2-3x+1,\phi_{25}\}=101^2\cdot 151^2$, and $\res\{x^2-3x+1,\phi_{50}\}=5^2\cdot3001^2$. Hence a K3 surface $X$ with intersection matrix $L_{3001}(1)$ of Picard lattice may have an automorphism $g$ such that $g^*\omega_X=\zeta_{50}\omega_X$ and $\trace(g^*_{NS(X)})=3$ (cf. \cite[Theorem A]{T}).
\end{example}

\begin{example} For $m=61$, since $61\mid f_{15}=610$, by Theorem \ref{main theorem1}, $X_{L_{61}(1)}$ has an anti-symplectic automorphism $g$ with $g^*\vert_{NS(X)}=(AB)^{15}$. Now note that $(AB)^6$ has trace $f_{11}+f_{13}=5f_6^2+2=322$. The resultant of $x^2-322x+1$ and $\Phi_5$ is $59^2\cdot 1741^2$ which cannot be divided by $61$ hence above $g$ is the generator of $X$. Indeed, if there were an automorphism $h$ such that $h^*\vert_{NS(X)}=(AB)^6$ and $h^*\omega_X=\zeta_5\omega_X$, then by \cite[Corollary 8.13]{K}, the characteristic polynomial of $h$ on $H^2(X,\Z)=(x^2-322x+1)(\Phi_5(x))^5$. For the discriminant group $A(L_{61}(1))\cong \Z_5\times \Z_{61}^2$, the two characteristic polynomials $x^2-322x+1$ and $\Phi_5=x^4+x^3+x^2+x+1$ have a common factor under the map $P\mapsto \bar{P}, ~ \Z[x]\rightarrow \F_p[x]$ for prime numbers $p=5$ and $61$ (cf. Section \ref{Resultants}). However, the resultant of them cannot be divided by $61$, a contradiction. 
\end{example}

\begin{example}(revisited Example \ref{example 100})
Consider Example \ref{example 100} in Section \ref{Examples}. Note that any divisor of $f_{20}$ divides $f_{100}$ by Theorem \ref{main theorem3} (2), but not necessarily divide $f_{50}$. If we assume that $m\mid f_{20}(=3\cdot 5\cdot 11\cdot 41)$ but $m\nmid f_{50}(=5^2\cdot 11\cdot 101\cdot 151\cdot 3001)$, then as in Example \ref{example 100}, $(l,k)=(25,4), (5,4), (5,20), (1,4),$ or $(1,100)$. However, $(l,k)=(25,4), (5,20)$ cannot occur. Indeed if $(l,k)=(25,4)$, then any prime divisor of $\disc(L_m(1))=5m^2$ should divide $\res\{(x-\lambda^4)(x-\lambda^{-4}), \phi_{25}\}=5^2(5f_{20}^4+5f_{20}^2+1)^2$ given in Lemma \ref{resultant with cyclotomic polynomials} (3). However, since $m\mid f_{20}$, it is not possible. Similarly, $(l,k)=(5,20)$ cannot occur since any prime divisor of $\disc(L_m(1))=5m^2$ cannot divide $\res\{(x-\lambda^{20})(x-\lambda^{-20}), \phi_{5}\}=5^2(5f_{20}^4+5f_{20}^2+1)^2$ given in Lemma \ref{resultant with cyclotomic polynomials} (1). Thus, $(l,k)=(5,4), (1,4),$ or $(1,100)$. 

For $(l,k)=(1,4)$, see Example \ref{example1}. Moreover, for example if $m=15$, then since $m$ cannot divide $\res\{(x-\lambda^{4})(x-\lambda^{-4}), \phi_{5}\}=5^2(5f_{4}^4+5f_{4}^2+1)^2=5^2\cdot 11^2\cdot 41^2$, hence $(l,k)=(5,4)$ cannot occur. Since $m=15\nmid f_4=4$, $(l,k)\neq (1,4)$, hence $g$ with $g^*\vert_{NS(X)}=(AB)^{100}$ is a symplectic generator of $\Aut(X)$.
\end{example}





\begin{thebibliography}{99}
\bibitem{FGvGvL}
Festi, D., Garbagnati, A., van Geemen, B., van Luijk, R.: The Cayley--Oguiso automorphism of positive entropy on a K3 surface. J. Mod. Dyn. {\bf 7} (1), 75--97 (2013)
\bibitem{Gr}
Gromov, M.: On the entropy of holomorphic maps. Enseign. Math. {\bf 49}, 217--235 (2003)
\bibitem{H}
Huybrechts D.: Lectures on K3 Surfaces. Cambridge University Press, Cambridge  (2016)
\bibitem{HKL} 
Hashimoto, K., Keum, J., Lee, K.: K3 surfaces with Picard number 2, Salem polynomials and Pell equation. J. Pure Appl. Algebra. {\bf 224}, 432-443 (2020)
\bibitem{K}
Kondo, S.: K3 surfaces. EMS Tracts Math., {\bf 32} EMS Publishing House, Berlin  (2020)
\bibitem{L1}
Lee, K.: Automorphisms of some K3 surfaces and Beauville involutions of their Hilbert schemes. Rend. Circ.
Mat. Palermo, II. Ser. {\bf 72} (4), 2443-2457 (2023)
\bibitem{O1}
Oguiso, K.: Free Automorphisms of positive entropy on smooth K\"ahler surfaces. In: Algebraic geometry in east Asia-Taipei 2011. Adv. Stud. Pure Math. {\bf 65}, 187-199 (2015)
\bibitem{T}
Takada, Y.: Dynamical degrees of automorphisms of K3 surfaces with Picard number 2. arXiv:2405.15195 (2024)
\bibitem{W}
Whitney, R.E.: Advanced problems and solutions. The Fibonacci quarterly, {\bf 10} (4), 413-422 (1972)
\bibitem{Yo}
Yomdin, Y.: Volume growth and entropy. Israel J. Math. {\bf 57}, 285--300 (1987)
\end{thebibliography}
\end{document}